\documentclass[smallextended]{svjour3}

\usepackage{amsmath}
\usepackage{amssymb}
\usepackage{url}

\usepackage{subfig}
\usepackage{siunitx}
\usepackage[ruled,vlined]{algorithm2e}

\newtheorem{thm}{Theorem}[section]

\newtheorem{lem}[thm]{Lemma}

\newtheorem{cor}[thm]{Corollary}

\begin{document}

\title{Multiscale methods with compactly supported radial basis functions for the Stokes problem on bounded domains \thanks{This work was supported by the Australian Research Council.}}

\titlerunning{Multiscale methods for the Stokes problem on bounded domains}

\author{A. Chernih \and Q. T. Le Gia}

\institute{A. Chernih
\at School of Mathematics and Statistics, University of New South Wales, Sydney NSW 2052,
Australia\\Tel.: +61-410-697411, Fax: +612 93857123 \\\email{andrew@andrewch.com}
\and
Q. T. Le Gia
\at School of Mathematics and Statistics, University of New South Wales\\\email{qlegia@maths.unsw.edu.au}
}

\maketitle
\begin{center}
\today
\end{center}

\begin{abstract}
In this paper, we investigate the application of radial basis functions (RBFs) for the approximation with collocation of the Stokes problem. The approximate solution is constructed in a multi-level fashion, each level using compactly supported radial basis functions with decreasing scaling factors. We use symmetric collocation and give sufficient conditions for convergence and consider stability analysis. Numerical experiments support the theoretical results.
\end{abstract}


\section{Introduction}

In this paper we investigate multiscale symmetric collocation approximation with Wendland compactly supported radial basis functions (RBFs) to solve the Stokes problem
\begin{eqnarray}
-\nu \Delta \mathbf{u} + \nabla p &=& \mathbf{f} \quad \mbox{ in } \,\, \Omega, \label{eqnStokes1} \\
\nabla \cdot \mathbf{u} &=& 0 \quad \mbox{ in } \,\, \Omega, \\
\mathbf{u} &=& \mathbf{g} \quad \mbox{ on } \,\, \partial \Omega, \label{eqnStokes3}
\end{eqnarray}
where the region $\Omega \subseteq \mathbb{R}^d$, the viscosity $\nu$, $\mathbf{f}: \Omega \rightarrow \mathbb{R}^d$ and $\mathbf{g}: \Omega \rightarrow \mathbb{R}^d$ are given and we seek an approximate solution to the velocity $\mathbf{u}: \Omega \rightarrow \mathbb{R}^d$ and the pressure $p: \Omega \rightarrow \mathbb{R}$.

Radial basis functions (RBFs) have been increasingly important in the area of approximation theory. For solving partial differential equations (PDEs), RBFs with meshless collocation for PDEs have been investigated in \cite{GieW06,Fas07} and for the Stokes problem in \cite{Wen09}. Matrix-valued, positive definite kernels have been studied in \cite{NarW94,Fus08,Fus08b,Low05,Low05b,Wen09}. Two excellent recent books covering practical and theoretical issues related to RBFs are \cite{Fas07} and \cite{Wen05}. We recall that a function $\Psi : \mathbb{R}^d \rightarrow \mathbb{R}$ is said to be \textit{radial} if there exists a function $\psi: [0,\infty) \rightarrow \mathbb{R}$ such that $\Psi(\mathbf{x}) = \psi(\|\mathbf{x}\|_2)$ for all $\mathbf{x} \in \mathbb{R}^d$, where $\|\cdot\|_2$ denoting the usual Euclidean norm in $\mathbb{R}^d$.  Then with a scaling factor $\delta> 0$, we can define a scaled RBF as
\begin{equation*}
\Psi_{\delta}(\mathbf{x}) = \delta^{-d} \psi\left(\frac{\|\mathbf{x}\|_2}{\delta}\right).
\end{equation*}

A practical issue that arises is that of which scale to use for the radial basis functions. A small scale will lead to a sparse and consequently well-conditioned linear system, but at the price of the approximation power. Conversely, a large scale will have better approximation power but at the price of a poorly-conditioned linear system.

Many examples may naturally exhibit multiple scales, for example, modelling fluid dynamics, where many different scales may be required. Of course, this comes at the price of having to select which scales to use in which regions but this is not the topic of this paper.

The multiscale algorithm investigated in this paper is constructed over multiple levels, in which the residual of the current stage is the target function for the next stage, and in each stage, RBFs with smaller support and with more closely spaced centres will be used as basis functions.

Such a multiscale algorithm for interpolation was first proposed in \cite{FloI96} and \cite{Sch96} but without any theoretical grounding. Theoretical convergence was proven in the case of the data points being located on a sphere \cite{LeSW10} and then extended to interpolation and approximation on bounded domains \cite{Wen10}.

In \cite{LeSW12} we can find an analysis of multiscale algorithms for RBF collocation of elliptic PDEs on the sphere.

The extension to considering the Stokes problem on a bounded domain changes the analysis significantly as matrix-valued kernels need to be considered with divergence-free approximation spaces. The scaled kernel, convergence and stability analysis are significant new contributions.

Our approach differs significantly from the literature, where for example in \cite{MasK06} the authors consider the incompressible Navier-Stokes equations in its weak formulation, and then decompose the velocity into coarse and fine scales.

In the next section we provide necessary background material regarding point sets and function spaces. Section \ref{SecnSymmStokes} describes our (single scale) symmetric collocation approximation and then Section \ref{SecnMultiSymmStokes} extends this to a multiscale algorithm and provides proofs of convergence. Section \ref{StokesCondNumber} provides an analysis of the stability of the approximations. Section \ref{SectionNumExperiments} provides numerical experiments to test the theoretical results.

\section{Preliminaries} \label{SectionPreliminaries}

In this paper, we will use (scaled) compactly supported radial basis functions to construct multiscale approximate solutions to the Stokes problem, that is, we form the solution over multiple levels. We will work with a given domain $\Omega \subseteq \mathbb{R}^d$. A kernel $\Phi: \Omega \times \Omega \rightarrow \mathbb{R}$ is also given.

At each level, we will have a finite point set $X \subseteq \Omega$. We will define the \textit{mesh norm} as
$$h_{X, \Omega} := \sup_{\mathbf{x} \in \Omega} \min_{\mathbf{x}_j \in X} \|\mathbf{x}-\mathbf{x}_j\|_2,$$
and the \textit{separation distance} as
$$q_{X} := \frac{1}{2} \min_{j \neq k} \|\mathbf{x}_j - \mathbf{x}_k\|_2,$$
which are measures of the uniformity of the points in $X$. Then for example, at each level $i$, we denote the mesh norm by $h_i$. The selection of point sets with mesh norms decreasing in a specific way will form one of the requirements for convergence of our algorithms and the separation distance will be used for the stability analysis.

We define the Sobolev spaces in the usual way. For a given domain, $\Omega \subseteq \mathbb{R}^d$, $k \in \mathbb{N}_0$, and $1 \leq p < \infty$, the Sobolev spaces $W_p^k(\Omega)$ consist of all $u$ with weak derivatives $D^{\alpha}u \in L_p(\Omega), \vert \alpha \vert \leq k.$ The semi-norms and norms are defined as
\begin{equation*}
\vert u \vert_{W_p^k(\Omega)} = \left( \sum_{\vert \alpha \vert = k} \| D^{\alpha}u\|_{L_p(\Omega)}^{p} \right)^{\frac{1}{p}} \quad \mbox{and} \quad \| u \|_{W_p^k(\Omega)} = \left( \sum_{\vert \alpha \vert \leq k} \| D^{\alpha}u\|_{L_p(\Omega)}^{p} \right)^{\frac{1}{p}}.
\end{equation*}
For $p=\infty$, these definitions become
\begin{equation*}
\vert u \vert_{W_{\infty}^k(\Omega)} = \sup_{\vert \alpha \vert = k} \|D^{\alpha}u\|_{L_{\infty}(\Omega)} \quad \mbox{and} \quad \| u \|_{W_{\infty}^k(\Omega)} = \sup_{\vert \alpha \vert \leq k} \|D^{\alpha}u\|_{L_{\infty}(\Omega)}.
\end{equation*}
Let $1 \leq p < \infty$, $k \in \mathbb{N}_0$, and $0 < s <1$. Then we can define the fractional Sobolev spaces $W_p^{k+s}(\Omega)$ as all $u$ for which the two norms
\begin{eqnarray*}
|u|_{W_p^{k+s}(\Omega)} &:=& \left( \sum_{|\alpha|=k} \int_{\Omega} \int_{\Omega} \frac{|D^{\alpha}u(\mathbf{x}) - D^{\alpha}u(\mathbf{y})|^p}{\|\mathbf{x} - \mathbf{y}\|_2^{d+ps}} \mathrm{d}\mathbf{x}\,\mathrm{d}\mathbf{y} \right)^{1/p} \\
\|u\|_{W_p^{k+s}(\Omega)} &:=& \left( \|u\|_{W_p^k(\Omega)}^p+ |u|_{W_p^{k+s}(\Omega)}^p \right)^{1/p}
\end{eqnarray*}
are finite. For the case $p=2$, we write $W_2^k(\Omega) = H^k(\Omega)$ and $L_2(\Omega) = W_2^0(\Omega)$.

The functions that we will be concerned with are defined on a bounded domain $\Omega$ with a Lipschitz boundary. As a result, there is an extension operator for functions defined in Sobolev spaces which is presented in the following lemma. For further details, we refer the reader to \cite{Ste70}
and \cite{DevS93}.
\begin{lem}
\label{lemExtensionOperator}
Suppose $\Omega \subseteq \mathbb{R}^d$ has a Lipschitz boundary. Then there is an extension mapping $E_S: H^{\tau}(\Omega) \rightarrow H^{\tau}(\mathbb{R}^d)$ defined for all non-negative real $\tau$ satisfying $E_S v\vert_{\Omega} = v \mbox{ for all } v \in H^{\tau}(\Omega)$ and
\begin{equation*}
\|E_Sv\|_{H^{\tau}(\mathbb{R}^d)} \leq C \|v\|_{H^{\tau}(\Omega)}.
\end{equation*}
\end{lem}
$C$ will denote a generic constant.
Since we also have $\|v\|_{H^{\tau}(\Omega)} \leq \|E_S v\|_{H^{\tau}(\mathbb{R}^d)}$, this means that when we need to consider the $H^{\tau}(\Omega)$ norms of the errors at each level, we can carry out our error analysis in the $H^{\tau}(\mathbb{R}^d)$-norm. This is advantageous, since we then have for $g \in H^{\tau}(\mathbb{R}^d)$
\begin{equation}
\label{eqnHtauNorm}
\|g\|_{H^{\tau}(\mathbb{R}^d)}^2 = \int\limits_{\mathbb{R}^d} |\widehat{g}(\boldsymbol{\omega})|^2 \left(1 + \|\boldsymbol{\omega}\|_2^2 \right)^{\tau} \mathrm{d}\boldsymbol{\omega},
\end{equation}
upon defining the Fourier transform as
$$\widehat{g}(\boldsymbol{\omega}) = \left(2\pi\right)^{-d/2} \int\limits_{\mathbb{R}^d} g(\boldsymbol{x}) e^{-i\mathbf{x}^T\boldsymbol{\omega}} \mathrm{d}\mathbf{x}.$$

At each level, we will also require a kernel $\Psi: \Omega \times \Omega \rightarrow \mathbb{R}$. We will use the Wendland compactly supported radial basis functions \cite{Wen05} with a (level-specific) scaling parameter $\delta>0$. We recall that for a given spatial dimension $d$ and smoothness parameter $k \in \mathbb{N}$, the Wendland functions are defined as
\begin{equation}
\psi_{\ell,k}(r) := \left\{
 \begin{array}{ll}
 \displaystyle{\frac{1}{\Gamma(k) \, 2^{k-1}} \int_r^1 s \,(1-s)^{\ell} (s^2-r^2)^{k-1} \mathrm{d}s} & \displaystyle{\mathrm{for }\hspace{0.1in} 0 \leq r \leq 1}, \\ \\[-2ex]
0 & \displaystyle{\mathrm{for }\hspace{0.1in} r > 1,}
\end{array} \right.
 \label{formulaWFsIntDefn}
\end{equation}
with $\ell := \lfloor \frac{d}{2} \rfloor + k + 1$. This choice of $\ell$ ensures the Wendland functions are positive definite. It is the reproducing kernel of a Hilbert space which is norm equivalent to the Sobolev space $H^{\frac{d+1}{2}+k}(\mathbb{R}^d)$.
For the Wendland basis functions, there exist two constants $0 < c_1 \leq c_2$, which depend on $d$ and $k$, such that their Fourier transforms satisfy \cite{Wen05}
\begin{equation}
\label{eqnWendlandFourierTequiv}
c_1\left(1 + \|\boldsymbol{\omega}\|_2^2 \right)^{-\frac{d}{2}-k-\frac{1}{2}} \leq \widehat{\Psi}_{\ell,k}(\boldsymbol{\omega}) \leq c_2 \left(1 + \|\boldsymbol{\omega}\|_2^2 \right)^{-\frac{d}{2}-k-\frac{1}{2}}, \quad \boldsymbol{\omega} \in \mathbb{R}^d.
\end{equation}

With a given kernel $\Psi$ and scaling factor $\delta >0$, we define the scaled kernel as
\begin{equation}
\Psi_{\delta}(\mathbf{x}) := \delta^{-d} \psi\left( \frac{\|\mathbf{x}\|_2}{\delta}\right). \label{eqnDefnScaledKernel}
\end{equation}

Appropriate selection of the scaling parameters will also prove to be one of the important ingredients for convergence of our multiscale algorithm.

We will need norm equivalence as stated in the following lemma from \cite{CheL12}.
\begin{lem}
\label{lemPhiDeltaHtauNormEquiv}
For every $\delta \in (0,\delta_a]$ and for all $g \in H^{\tau}(\mathbb{R}^d)$, there exist constants $0 < c_3 \leq c_4$ such that
\begin{equation*}
c_3 \|g\|_{\Phi_{\delta}} \leq \|g\|_{H^{\tau}(\mathbb{R}^d)} \leq c_4 \delta^{-\tau} \|g\|_{\Phi_{\delta}}.
\end{equation*}
\end{lem}

As we will be working with vectors, in particular for $\mathbf{u}$, we will need to define vector-valued Sobolev spaces in the usual way as
\[
\mathbf{H}^{\tau}(\Omega) := H^{\tau}(\Omega) \times \ldots \times H^{\tau}(\Omega),
\]
with norm
\begin{equation}
\|\mathbf{f}\|_{\mathbf{H}^{\tau}(\Omega)} := \left( \sum_{j=1}^d \|f_j\|_{H^{\tau}(\Omega)}^2 \right)^{1/2}. \label{defnVectorSobolevNorm}
\end{equation}
Now we define divergence-free approximation spaces in $\Omega$ and in $\mathbb{R}^d$. With the divergence of $\mathbf{u}: \Omega \rightarrow \mathbb{R}^d$ defined as
\[
\nabla \cdot \mathbf{u} := \sum_{j=1}^d \partial_j u_j\,\, ,
\]
we define
\[
\mathbf{H}^{\tau}(\Omega; \mbox{div}) := \left\{ \mathbf{u} \in \mathbf{H}^{\tau}(\Omega): \nabla \cdot \mathbf{u} = 0 \right\},
\]
and
\[
\widetilde{\mathbf{H}}^{\tau}(\mathbb{R}^d;\mbox{div}) := \left\{ \mathbf{f} \in \mathbf{H}^{\tau}(\mathbb{R}^d;\mbox{div}): \int_{\mathbb{R}^d} \frac{\|\widehat{\mathbf{f}}(\boldsymbol{\omega})\|_2^2}{\|\boldsymbol{\omega}\|_2^2} \left(1 + \|\boldsymbol{\omega}\|_2^2 \right)^{\tau+1} \mathrm{d} \boldsymbol{\omega} < \infty \right\},
\]
with norm
\[
\|\mathbf{f}\|^2_{\widetilde{\mathbf{H}}^{\tau}(\mathbb{R}^d;\scriptstyle \text{div})} := (2\pi)^{-d/2} \int_{\mathbb{R}^d} \frac{\|\widehat{\mathbf{f}}(\boldsymbol{\omega})\|_2^2}{\|\boldsymbol{\omega}\|_2^2} \left(1 + \|\boldsymbol{\omega}\|_2^2 \right)^{\tau+1} \mathrm{d} \boldsymbol{\omega}.
\]
We note that $\widetilde{\mathbf{H}}^{\tau}(\mathbb{R}^d;\mbox{div})$ is a subspace of $\mathbf{H}^{\tau}(\mathbb{R}^d;\mbox{div})$. We will also need that for $\Omega \subseteq \mathbb{R}^d$ being a simply connected domain with $C^{\lceil\tau \rceil,1}$ boundary for $d=2,3$ and with $\tau \geq 0$, there exists a continuous operator
\[
\widetilde{\mathbf{E}}_{\scriptstyle \text{div}}: \mathbf{H}^{\tau}(\Omega;\mbox{div}) \rightarrow \widetilde{\mathbf{H}}^{\tau}(\mathbb{R}^d;\mbox{div}),
\]
such that $\widetilde{\mathbf{E}}_{\scriptstyle \text{div}}\mathbf{u}\vert \Omega  = \mathbf{u}$ for all $\mathbf{u} \in \mathbf{H}^{\tau}(\Omega; \mbox{div})$ \cite[Proposition 3.8]{Wen09}. For $d=3$, this operator is defined as
\begin{equation}
\label{eqnEdivDefn}
\widetilde{\mathbf{E}}_{\scriptstyle \text{div}}\mathbf{u} :=
\nabla \times E_S T \mathbf{u},
\end{equation}
where $E_S$ is the extension operator defined in Lemma \ref{lemExtensionOperator} and $\mathbf{v}=T\mathbf{u}$  is the unique solution of the boundary
value problem
\[
 \mathbf{u} = \nabla \times\mathbf{v}, \;\;\nabla \cdot \mathbf{v} = 0 \text{ in } \Omega,\qquad
 \mathbf{v} \cdot \mathbf{n} = 0 \text{ on } \partial \Omega.
\]
For $d=2$, formula \eqref{eqnEdivDefn} is replaced by
\begin{equation}
\label{eqnEdivDefnd2}
\widetilde{\mathbf{E}}_{\scriptstyle \text{div}}\mathbf{u} :=
  \text{curl}\; E_S T \mathbf{u},
\end{equation}
where $T \mathbf{u} = \psi$ with $\mathbf{u} = \text{curl}\;\psi
= (\partial_y \psi, -\partial_x \psi)$.

Using the idea from \cite[Lemma 4]{Fus08b} and interpolation
of operators (see \cite[Proposition (14.1.5)]{BreS08}) we can show that
the operator $T: \mathbf{H}^{\tau}(\Omega; \mbox{div}) \rightarrow \mathbf{H}^{\tau+1}(\Omega)$, with $\tau = k+\theta$ for $k \in \mathbb{N}_0$ and $\theta \in [0,1]$, is bounded.

To measure the pressure, which is determined only up to a constant, we will use the norm
\[
\|p\|_{H^{\tau}(\Omega)/\mathbb{R}} := \inf_{c \in \mathbb{R}} \|p+c\|_{H^{\tau}(\Omega)}.
\]

We follow \cite{GieW06} to define Sobolev norms and the mesh norm on the boundary. We assume that $\partial \Omega \subseteq \cup_{j=1}^K V_j$, where $V_j \subseteq \mathbb{R}^d$ are open sets. The sets $V_j$ are images of $C^{k,s}-$diffeomorphisms
 \begin{equation*}
 \varphi_j : B \rightarrow V_j,
 \end{equation*}
 where $B = B(0,1)$ denotes the unit ball in $\mathbb{R}^{d-1}$. If $\{w_j\}$ is a partition of unity with respect to $\{V_j\}$, then the Sobolev norms on $\partial \Omega$ can be defined as
 \begin{equation*}
 \|u\|_{W_p^{\mu}(\partial \Omega)}^p := \sum_{j=1}^K \| (u w_j) \circ \varphi_j \|_{W_p^{\mu}(B)}^p.
 \end{equation*}
 The mesh norm on the boundary can be defined as
 \begin{equation*}
 h_{X,\partial \Omega} := \max_{1 \leq j \leq K} h_{T_j,B},
 \end{equation*}
 with $T_j := \varphi_j^{-1}(X \cap V_j) \subseteq B$.
 Finally, we will need the following ``sampling" inequalities, which are valid for both scalars and vectors \cite{NarWW05,NarWW06,Wen09}.
\begin{thm}
Let $\Omega \subseteq \mathbb{R}^d$ be a bounded domain with Lipschitz boundary. Let $\tau > d/2$. Let $X \subseteq \Omega$ be a discrete set having mesh norm $h$ sufficiently small. For each $w \in H^{\tau}(\Omega)$ with $w|X=0$ we have for $0 \leq \sigma \leq \tau$ that
\begin{equation}
\|w\|_{H^{\sigma}(\Omega)} \leq C h^{\tau-\sigma} \|w\|_{H^{\tau}(\Omega)}. \label{eqnSampling}
\end{equation}
\end{thm}

\begin{thm}
Let $\tau = k+s >d/2$. Let $\Omega \subseteq \mathbb{R}^d$ be a bounded domain having $C^{k,s}$ smooth boundary. Let $X \subseteq \partial \Omega$ be a discrete set with $h$ sufficiently small. Then there is a positive constant C such that for all $w \in H^{\tau}(\Omega)$ with $w|X=0$ we have for $0 \leq \sigma \leq \tau - 1/2$ that
\begin{equation}
\|w\|_{H^{\sigma}(\partial \Omega)} \leq C h^{\tau-1/2-\sigma} \|w\|_{H^{\tau}(\Omega)}. \label{eqnSamplingBdy}
\end{equation}
\end{thm}
We also define a matrix-valued function $\mathbf{\Phi}: \mathbb{R}^d \rightarrow \mathbb{R}^{n \times n}$ as being \textit{positive definite} if it is even, so $\mathbf{\Phi}(-\mathbf{x}) = \mathbf{\Phi}(\mathbf{x})$, symmetric, so $\mathbf{\Phi}(\mathbf{x}) = \mathbf{\Phi}(\mathbf{x})^T$, and satisfies
\[
\sum_{j,k=1}^n \boldsymbol{\alpha}_j^T \mathbf{\Phi}(\mathbf{x}_j - \mathbf{x}_k) \boldsymbol{\alpha}_k > 0,
\]
for all pairwise distinct $\mathbf{x}_j \in \mathbb{R}^d$ and all $\boldsymbol{\alpha}_j \in \mathbb{R}^n$ such that not all $\boldsymbol{\alpha}_j$ are vanishing.

\section{Symmetric collocation approximation} \label{SecnSymmStokes}

We will first consider a single-scale approximant to the combined velocity and pressure vector $\mathbf{v} := (\mathbf{u},p): \mathbb{R}^d \rightarrow \mathbb{R}^{d+1}$, following \cite{NarW94,Fus08,Wen09}. Then \eqref{eqnStokes1}-\eqref{eqnStokes3} become
\begin{eqnarray}
\left(L \mathbf{v} \right)_i := -\nu \sum_{j=1}^d \partial_{jj} v_i + \partial_i v_{d+1} &=& f_i \quad \mbox{ in } \,\, \Omega, \label{eqnModStokes1} \\
\sum_{j=1}^d \partial_j v_j &=& 0 \quad \mbox{ in } \,\, \Omega, \label{eqnModStokes2}\\
v_i &=& g_i \quad \mbox{ on } \,\, \partial \Omega, \label{eqnModStokes3}
\end{eqnarray}
where $1 \leq i \leq d$. We seek a meshfree, kernel-based collocation method with an analytically divergence-free approximation space. We use the notation $\psi_{\tau+1}$ and $\psi_{\tau-1}$ to denote the functions to be used in our matrix-valued kernel. We will mainly be interested in the case where both $\psi_{\tau+1}$ and $\psi_{\tau-1}$ are Wendland functions which, for a given spatial dimension $d$, have native space norms equivalent to the Sobolev spaces $H^{\tau+1}(\mathbb{R}^d)$ and $H^{\tau-1}(\mathbb{R}^d)$ respectively. Their Fourier transforms satisfy
\begin{equation}
\label{eqnFourierUpsilon}
c_{1,\tau+1}(1+\|\boldsymbol{\omega}\|_2^2)^{-\tau-1} \leq \widehat{\psi_{\tau+1}}(\|\boldsymbol{\omega}\|_2) \leq c_{2,\tau+1}(1+\|\boldsymbol{\omega}\|_2^2)^{-\tau-1},
\end{equation}
and
\begin{equation}
\label{eqnFourierPsi}
c_{1,\tau-1}(1+\|\boldsymbol{\omega}\|_2^2)^{-\tau+1} \leq \widehat{\psi_{\tau-1}}(\|\boldsymbol{\omega}\|_2) \leq c_{2,\tau-1}(1+\|\boldsymbol{\omega}\|_2^2)^{-\tau+1},
\end{equation}
and we define $\bar{C}_1 := \min(c_{1,\tau+1},c_{1,\tau-1})$ and $\bar{C}_2 := \max(c_{2,\tau+1},c_{2,\tau-1})$. Then we define the matrix-valued kernel
\begin{equation}
\label{eqnDefnPhi}
\boldsymbol{\Phi} := \begin{pmatrix} \boldsymbol{\Psi}_{\tau+1} & 0 \\ 0 & \psi_{\tau-1} \end{pmatrix}: \mathbb{R}^d \rightarrow \mathbb{R}^{(d+1) \times (d+1)},
\end{equation}
where $\boldsymbol{\Psi}_{\tau+1} := (-\Delta \mathbf{I} + \nabla \nabla^T)\psi_{\tau+1}$ with $\mathbf{I}$ denoting the identity matrix. We note that $\boldsymbol{\Psi}_{\tau+1}$ is also positive definite \cite{NarW94} and hence due to the tensor product construction of $\mathbf{\Phi}$, it is positive definite as well. This choice for $\boldsymbol{\Psi}_{\tau+1}$ is known to lead to divergence-free interpolants \cite{NarW94}. We also note that
\begin{equation}
\widehat{\mathbf{\Psi}_{\tau+1}}(\boldsymbol{\omega}) = \left( \|\boldsymbol{\omega}\|_2^2 \mathbf{I} - \boldsymbol{\omega}\boldsymbol{\omega}^T\right) \widehat{\psi_{\tau+1}}(\boldsymbol{\omega}). \label{eqnFourierUpsilonDefn}
\end{equation}
We will consider the case where the collocation points are the same as the RBF centres. We denote the interior centres by $X_1 := \{\mathbf{x}_1,\ldots,\mathbf{x}_N\}$ and the boundary centres by $X_2 :=  \{\mathbf{x}_{N+1},\ldots,\mathbf{x}_M\}$ and their union by $X = X_1 \cup X_2$, with mesh norms $h_1$ and $h_2$ respectively. Since \eqref{eqnModStokes2} is automatically satisfied, this means that our approximant and collocation conditions will consist of $dN$ terms from \eqref{eqnModStokes1} and $d(M-N)$ terms from \eqref{eqnModStokes3}. Then with $L^{y}$ denoting the operator $L$ acting as a function of the second argument, applied to rows of $\boldsymbol{\Phi}$, our approximant takes the form
\begin{equation}
\mathbf{S}_{X}\mathbf{v}(\mathbf{x}) = \sum_{i=1}^d \sum_{j=1}^N \alpha_{i,j}  \left( L^{y} \boldsymbol{\Phi}\left(\mathbf{x}-\mathbf{x}_j\right) \right)_i + \sum_{i=1}^d  \sum_{j=N+1}^M \alpha_{i,j} \boldsymbol{\Phi}\left(\mathbf{x}-\mathbf{x}_j\right)_i, \label{eqnStokesApproxDefn} \\
\end{equation}
where the notation $\boldsymbol{\Phi}_i$ means column $i$ of the matrix $\boldsymbol{\Phi}$ and $\mathbf{S}_{X}\mathbf{v} = (\mathbf{S}_{X}\mathbf{u},S_X p)$.
The coefficients $\alpha_{i,j}$, $1\leq i \leq d$, $1 \leq j \leq M$ are determined by the collocation conditions
\begin{eqnarray}
\left( L \mathbf{S}_{X}\mathbf{v}(\mathbf{x}_j) \right)_i &=& f_i(\mathbf{x}_j), \quad 1 \leq i \leq d, \,\, j=1,\ldots,N, \label{eqnColl1}\\
\left( \mathbf{S}_{X}\mathbf{v}(\mathbf{x}_j) \right)_i &=& g_i(\mathbf{x}_j), \quad 1 \leq i \leq d, \,\, j=N+1,\ldots,M. \label{eqnColl2}
\end{eqnarray}
From \cite{Fus08,Wen09}, we know that if $\psi_{\tau+1},\psi_{\tau-1}$ are positive definite and if $\mathbf{\Psi}_{\tau+1} \in W_1^2(\mathbb{R}^d)\cap C^2(\mathbb{R}^d)$, then the native space of the kernel $\mathbf{\Phi}$ given by \eqref{eqnDefnPhi} is
\[
\mathcal{N}_{\mathbf{\Phi}}(\mathbb{R}^d) = \mathcal{N}_{\mathbf{\Psi}_{\tau+1}}(\mathbb{R}^d) \times \mathcal{N}_{\psi_{\tau-1}}(\mathbb{R}^d),
\]
with norm
\begin{eqnarray}
\|\mathbf{f}\|^2_{\mathcal{N}_\mathbf{\Phi}(\mathbb{R}^d)} &=& \|\mathbf{f_u}\|^2_{\mathcal{N}_{\mathbf{\Psi}_{\tau+1}}(\mathbb{R}^d)} + \|f_p\|^2_{\mathcal{N}_{\psi_{\tau-1}}(\mathbb{R}^d)} \nonumber \\
&=& (2\pi)^{-d/2} \int_{\mathbb{R}^d} \Bigg[ \frac{\|\widehat{\mathbf{f}}_{\mathbf{u}}(\boldsymbol{\omega})\|_2^2}{\|\boldsymbol{\omega}\|_2^2 \widehat{\psi_{\tau+1}}(\boldsymbol{\omega})} + \frac{\vert \widehat{f}_p(\boldsymbol{\omega}) \vert^2}{\widehat{\psi_{\tau-1}}(\boldsymbol{\omega})} \Bigg] \mathrm{d} \boldsymbol{\omega}, \label{eqnStokesPhiNorm}
\end{eqnarray}
where $\mathbf{f} = (\mathbf{f}_{\mathbf{u}},f_p)^T$ with $\mathbf{f}_{\mathbf{u}}: \mathbb{R}^d \rightarrow \mathbb{R}^d$ and $f_p: \mathbb{R}^d \rightarrow \mathbb{R}$. We recall that the generalised interpolant satisfies \cite[Chapter 16]{Wen05}
\[
\|\mathbf{Ev} - \mathbf{S}_{X}\mathbf{Ev}\|_{\mathcal{N}_{\mathbf{\Phi}}(\mathbb{R}^d)} \leq \|\mathbf{Ev}\|_{\mathcal{N}_{\mathbf{\Phi}}(\mathbb{R}^d)}.
\]
With \eqref{eqnFourierUpsilon} and \eqref{eqnFourierPsi}, we define the extension operator for the velocity-pressure vector $\mathbf{v}$ as
\begin{equation}
\mathbf{Ev} := \left(\widetilde{\mathbf{E}}_{\scriptstyle \text{div}}\mathbf{u},E_{\scriptstyle \text{S}} p \right), \label{eqnStokesExtension}
\end{equation}
where $E_{\scriptstyle \text{S}}$ is the classical Stein extension operator as defined in Lemma \ref{lemExtensionOperator}. Then the native space of our approximant given by \eqref{eqnStokesApproxDefn} is
\[
\mathbf{E}: \mathbf{H}^{\tau}(\Omega;\mbox{div}) \times H^{\tau-1}(\Omega) \rightarrow \mathcal{N}_{\mathbf{\Phi}}(\mathbb{R}^d) = \widetilde{\mathbf{H}}^{\tau}(\mathbb{R}^d;\mbox{div}) \times H^{\tau-1}(\mathbb{R}^d).
\]
Once again we can define interpolants with scaled kernels. In this case, we define the matrix-valued kernel
\begin{equation}
\label{eqnDefnPhiDelta}
\boldsymbol{\Phi}_{\delta} := \begin{pmatrix} \boldsymbol{\Psi}_{\tau+1,\delta} & 0 \\ 0 & \psi_{\tau-1,\delta} \end{pmatrix}: \mathbb{R}^d \rightarrow \mathbb{R}^{(d+1) \times (d+1)},
\end{equation}
where $\boldsymbol{\Psi}_{\tau+1,\delta} := (-\Delta \mathbf{I} + \nabla \nabla^T)\psi_{\tau+1,\delta}$ and the scaled basis functions are defined as in \eqref{eqnDefnScaledKernel}. Then the native space of the kernel $\mathbf{\Phi}_{\delta}$ is given by
\[
\mathcal{N}_{\mathbf{\Phi}_{\delta}}(\mathbb{R}^d) = \mathcal{N}_{\mathbf{\Psi}_{\tau+1,\delta}}(\mathbb{R}^d) \times \mathcal{N}_{\psi_{\tau-1,\delta}}(\mathbb{R}^d),
\]
with norm
\begin{eqnarray}
\|\mathbf{f}\|^2_{\mathcal{N}_{\mathbf{\Phi}_{\delta}}(\mathbb{R}^d)} &=& \|\mathbf{f_u}\|^2_{\mathcal{N}_{\mathbf{\Psi}_{\tau+1,\delta}}(\mathbb{R}^d)} + \|f_p\|^2_{\mathcal{N}_{\psi_{\tau-1,\delta}}(\mathbb{R}^d)} \nonumber \\
&=& (2\pi)^{-d/2} \int_{\mathbb{R}^d} \Bigg[ \frac{\|\widehat{\mathbf{f}}_{\mathbf{u}}(\boldsymbol{\omega})\|_2^2}{\|\boldsymbol{\omega}\|_2^2 \widehat{\psi_{\tau+1,\delta}}(\boldsymbol{\omega})} + \frac{\vert \widehat{f}_p(\boldsymbol{\omega}) \vert^2}{\widehat{\psi_{\tau-1,\delta}}(\boldsymbol{\omega})} \Bigg] \mathrm{d} \boldsymbol{\omega}. \label{eqnStokesPhiDeltaNorm}
\end{eqnarray}
We will need norm equivalence as stated in the following lemma.
\begin{lem}
\label{lemStokesNormEquiv}
For every $\delta \in (0,\delta_a]$ where $\psi_{\tau+1}$ and $\psi_{\tau-1}$ generate $H^{\tau+1}(\mathbb{R}^d)$ and $H^{\tau-1}(\mathbb{R}^d)$ respectively, we have $\mathcal{N}_{\mathbf{\Phi}_{\delta}}(\mathbb{R}^d) = \mathcal{N}_{\mathbf{\Phi}}(\mathbb{R}^d)$ and for every $\mathbf{f} \in \mathcal{N}_{\mathbf{\Phi}}(\mathbb{R}^d)$ there exist positive constants $c_3$ and $c_4$ such that
\[
c_3 \|\mathbf{f}\|_{\mathcal{N}_{\mathbf{\Phi}_{\delta}}(\mathbb{R}^d)} \leq \|\mathbf{f}\|_{\mathcal{N}_{\mathbf{\Phi}}(\mathbb{R}^d)} \leq c_4 \delta^{-\tau-1} \|\mathbf{f}\|_{\mathcal{N}_{\mathbf{\Phi}_{\delta}}(\mathbb{R}^d)}.
\]
\end{lem}
\begin{proof}
With $\mathbf{f} = (\mathbf{f}_{\mathbf{u}},f_p)^T$,
by using arguments similar to \cite[Lemma 2.2]{CheL12} we have
\[
c_5 \|f_p\|_{\mathcal{N}_{\psi_{\tau-1,\delta}}(\mathbb{R}^d)}
\le \|f_p\|_{\mathcal{N}_{\psi_{\tau-1}}(\mathbb{R}^d)}
\le c_6 \delta^{-\tau-1}
  \|f_p\|_{\mathcal{N}_{\psi_{\tau-1,\delta}}(\mathbb{R}^d)},
\]
where $c_5 := c_{1,\tau-1}\min(1,\delta^{-\tau-1})$ and
$c_6 := c_{2,\tau-1}$.
Similarly, we can show
\[
c_7\|\mathbf{f_u}\|_{\mathcal{N}_{\mathbf{\Psi}_{\tau+1,\delta}}(\mathbb{R}^d)}
\le \|\mathbf{f_u}\|_{\mathcal{N}_{\mathbf{\Psi}_{\tau+1}}(\mathbb{R}^d)}
\le c_8 \delta^{-\tau-1}\|\mathbf{f_u}\|_{\mathcal{N}_{\mathbf{\Psi}_{\tau+1,\delta}}(\mathbb{R}^d)},
\]
where $c_7 := c_{1,\tau+1}\min(1,\delta_a^{-\tau-1})$ and
$c_8 := c_{2,\tau+1}$.
With \eqref{eqnStokesPhiDeltaNorm} and
setting $c_3 := \min(c_5,c_7)$ and $c_4 := \max(c_6,c_8)$, we get the final
result.
\end{proof}

We require one further result from \cite{Tem77}.
\begin{thm}
Let $m \in \mathbb{N}_0$ and let $\Omega \subseteq \mathbb{R}^d$ be a $C^{m+1,1}$ smooth domain with outer normal vector $\mathbf{n}$. For each $\mathbf{f} \in \mathbf{H}^m(\Omega)$ and $\mathbf{g} \in \mathbf{H}^{m+3/2}(\partial \Omega)$ with $\int_{\partial \Omega} \mathbf{g} \cdot \mathbf{n} \, \mathrm{d}S = 0$, the nonhomogeneous Stokes problem \eqref{eqnStokes1}-\eqref{eqnStokes3} has a unique solution $\mathbf{u} \in \mathbf{H}^{m+2}(\Omega)$ and $p \in H^{m+1}(\Omega)$ and
\begin{equation}
\label{eqnTem}
\|\mathbf{u}\|_{\mathbf{H}^{m+2}(\Omega)} + \|p\|_{H^{m+1}(\Omega)/\mathbb{R}} \leq C \left(\|\mathbf{f}\|_{\mathbf{H}^m(\Omega)} + \|\mathbf{g}\|_{\mathbf{H}^{m+3/2}(\partial \Omega)}\right)
\end{equation}
\end{thm}

\begin{thm}
\label{thmStokesL2}
Let $\tau > 2 + d/2$ with $d=2,3$. Assume that $\Omega \subseteq \mathbb{R}^d$ is a bounded, simply connected region with a $C^{\lceil \tau \rceil,1}$ boundary. Let $\mathbf{f} \in \mathbf{H}^{\tau-2}(\Omega)$ and $\mathbf{g} \in \mathbf{H}^{\tau-1/2}(\partial \Omega)$ satisfy $\int_{\partial \Omega} \mathbf{g} \cdot \mathbf{n} \, \mathrm{d} S = 0$. Suppose the kernel $\mathbf{\Phi}$ is chosen such that $\mathcal{N}_{\mathbf{\Phi}}(\mathbb{R}^d) = \widetilde{\mathbf{H}}^{\tau}(\mathbb{R}^d;\mathrm{div}) \times H^{\tau-1}(\mathbb{R}^d)$. Then the approximation $\mathbf{S}_X \mathbf{v}$ given by \eqref{eqnStokesApproxDefn} to the Stokes problem \eqref{eqnStokes1}-\eqref{eqnStokes3} satisfies the error bound
\begin{equation}
\|\mathbf{v} - \mathbf{S}_{X}\mathbf{v}\|_{\mathbf{L}_2(\Omega)} \leq C \, \bar{h}^{\tau -2} \|\mathbf{Ev} - \mathbf{S}_{X}\mathbf{Ev} \|_{\mathcal{N}_{\mathbf{\Phi}}(\mathbb{R}^d)},
\end{equation}
where $\bar{h} := \max(h_1,h_2)$ and the extension operator $\mathbf{E}$ is given by \eqref{eqnStokesExtension}.
\end{thm}
\begin{proof}
With the definition of the Sobolev space norms in \eqref{defnVectorSobolevNorm} and assuming that we choose the representer for the pressure $p$ such that $\|p\|_{H^1(\Omega)/\mathbb{R}} = \|p\|_{H^1(\Omega)}$ gives
\begin{eqnarray}
\|\mathbf{v}-\mathbf{S}_{X}\mathbf{v}\|_{\mathbf{L}_2(\Omega)} &\leq& \|\mathbf{u}-\mathbf{S}_{X}\mathbf{u}\|_{\mathbf{L}_2(\Omega)} + \|p-S_{X}p\|_{L_2(\Omega)} \nonumber \\
&\leq& \|\mathbf{u}-\mathbf{S}_{X}\mathbf{u}\|_{\mathbf{H}^2(\Omega)} + \|p-S_{X}p\|_{H^1(\Omega)} \nonumber \\
&=& \|\mathbf{u}-\mathbf{S}_{X}\mathbf{u}\|_{\mathbf{H}^2(\Omega)} + \|p-S_{X}p\|_{H^1(\Omega)/\mathbb{R}} \nonumber \\
&\leq& C\|L\mathbf{v} - L\mathbf{S}_{X}\mathbf{v}\|_{\mathbf{L}_2(\Omega)} + \|\mathbf{u}-\mathbf{S}_{X}\mathbf{u}\|_{\mathbf{H}^{3/2}(\partial \Omega)}, \label{eqnL2int}
\end{eqnarray}
where the last line follows from \eqref{eqnTem} applied to $\mathbf{v}-\mathbf{S}_X\mathbf{v}$ with $m=0$. We now extend the function $\mathbf{v}$ to $\mathbf{Ev} \in \widetilde{\mathbf{H}}^{\tau}(\mathbb{R}^d) \times H^{\tau-1}(\mathbb{R}^d)$ and note that the generalised interpolant $\mathbf{S}_{X}\mathbf{v}$ coincides with $\mathbf{S}_{X}\mathbf{Ev}$.
We now consider the two terms in the right hand side of \eqref{eqnL2int} separately. From \eqref{eqnSampling} and with \cite[p.3173]{Wen09},
we have
\[
\|L\mathbf{v} - L\mathbf{S}_{X}\mathbf{v}\|_{\mathbf{L}_2(\Omega)} \leq C h_1^{\tau-2} \|\mathbf{Ev}-\mathbf{S}_{X}\mathbf{Ev}\|_{\mathcal{N}_{\mathbf{\Phi}}(\mathbb{R}^d)}.
\]
From \eqref{eqnSamplingBdy}, we have
\begin{equation}
\|\mathbf{u}-\mathbf{S}_{X}\mathbf{u}\|_{\mathbf{H}^{3/2}(\partial \Omega)} \leq C h_2^{\tau-2} \|\mathbf{u}-\mathbf{S}_{X}\mathbf{u}\|_{\mathbf{H}^{\tau}(\Omega)}. \label{eqnVelBoundaryError}
\end{equation}
Now we can write
\begin{eqnarray*}
\|\mathbf{u}-\mathbf{S}_{X}\mathbf{u}\|_{\mathbf{H}^{\tau}(\Omega)} &\leq& \|\mathbf{u} - \mathbf{S}_{X}\mathbf{u}\|_{\mathbf{H}^{\tau}(\Omega)} + \|p - S_{X}p\|_{H^{\tau-1}(\Omega)} \\
&\leq& \|\mathbf{\widetilde{E}}_{\scriptstyle \text{div} } \mathbf{u} - \mathbf{S}_{X}\mathbf{\widetilde{E}}_{\scriptstyle \text{div}} \mathbf{u}\|_{\widetilde{\mathbf{H}}^{\tau}(\mathbb{R}^d;{\scriptstyle \text{div}})} + \|E_Sp -  S_X E_{\scriptstyle S} p\|_{H^{\tau-1}(\mathbb{R}^d)} \\
&\leq& C \|\mathbf{Ev} - \mathbf{S}_X \mathbf{Ev} \|_{\mathcal{N}_{\mathbf{\Phi}}(\mathbb{R}^d)},
\end{eqnarray*}
and the stated result follows.
\end{proof}

\section{Multiscale symmetric collocation approximation} \label{SecnMultiSymmStokes}
We can now formally state our multiscale algorithm for the symmetric collocation solution of \eqref{eqnStokes1}-\eqref{eqnStokes3} which is stated as Algorithm \ref{AlgSymmStokes}. To simplify notation, we write $\mathbf{S}_i \mathbf{v} = \mathbf{S}_{X_i}\mathbf{v}$ and $\mathbf{\Phi}_i = \mathbf{\Phi}_{\delta_i}$ and denote the mesh norms for the interior and boundary collocation points at level $i$ as $h_{1,i}$ and $h_{2,i}$ respectively.
\begin{algorithm}[!htbp]
\KwData{\hspace{0.45in} $n$: number of levels \\ \hspace{0.45in} $X_i:=\{X_{1,i},X_{2,i}\}_{i=1}^{n}$: the interior and boundary collocation \\ \hspace{0.45in}points for each level $i$, with mesh norms at each level given by\\ \hspace{0.45in} $\{h_{1,i}, h_{2,i}\}_{i=1}^n$ satisfying $c \mu \bar{h}_i \leq \bar{h}_{i+1} \leq \mu \bar{h}_i$, where \\ \hspace{0.45in} $\bar{h}_i := \max(h_{1,i},h_{2,i})$ with fixed $\mu \in (0,1), c \in (0,1]$ and \\ \hspace{0.45in} $\bar{h}_1$ sufficiently small \\ \hspace{0.45in} $\{\delta_i\}_{i=1}^n:$ the scale parameters to use at each level, satisfying\\ \hspace{0.45in} $\delta_i = \beta \bar{h}_i^{1-3/(\tau+1)}$, $\beta$ is a fixed constant. }

\Begin{
Set $\mathbf{M}_{0}\mathbf{v}=\mathbf{0}, \mathbf{f}_0 = \mathbf{f}, \mathbf{g}_0=\mathbf{g}$ \\
\For{$i =1,2,\ldots,n$}{
With the scaled kernel $\mathbf{\Phi}_{i}$, solve the symmetric collocation linear system
\begin{eqnarray*}
\left( L \mathbf{S}_{i}\mathbf{v}(\mathbf{x}) \right)_j &=& f_{i-1,j}(\mathbf{x}), \quad 1 \leq j \leq d, \,\, \mathbf{x}\in X_{1,i} \\
\left( \mathbf{S}_{i}\mathbf{v}(\mathbf{x}) \right)_j &=& g_{i-1,j}(\mathbf{x}), \quad 1 \leq j \leq d, \,\, \mathbf{x} \in X_{2,i}.
\end{eqnarray*}

Update the solution and residual according to
\begin{eqnarray*}
\mathbf{M}_{i}\mathbf{v} &=& \mathbf{M}_{i-1}\mathbf{v} + \mathbf{S}_{i}\mathbf{v} \\
\mathbf{f}_i &=& \mathbf{f}_{i-1} - L \mathbf{S}_{i}\mathbf{v}  \\
\mathbf{g}_i &=& \mathbf{g}_{i-1} - \mathbf{S}_{i}\mathbf{v}
\end{eqnarray*}
}
}
\KwResult{Approximate solution at level $n$, $\mathbf{M}_{n}\mathbf{v}$ \\ The error at level $n$, $\mathbf{e}_n := \mathbf{v} - \mathbf{M}_{n}\mathbf{v}$.}
\caption{Multiscale symmetric collocation approximation to the Stokes problem \label{AlgSymmStokes}}
\end{algorithm}

We require a technical lemma regarding the error in the estimation of the velocity $u$.
\begin{lem}
\label{lemTechErrBoundU}
Let $d=2,3$. Assume that $\mathbf{u} \in \mathbf{H}^{\tau}(\Omega; \mathrm{div})$ with $\tau>0$ and let $\widetilde{\mathbf{E}}_{\scriptstyle \mathrm{div}}$ be defined by \eqref{eqnEdivDefn} for $d=2,3$. Then we have the following bound
\begin{equation*}
\int\limits_{\mathbb{R}^d} \frac{\Bigl\| \widehat{\widetilde{\mathbf{E}}_{\scriptstyle \mathrm{div}} \mathbf{u}}(\boldsymbol{\omega}) \Bigr\|_2^2}{\|\boldsymbol{\omega}\|_2^2} \mathrm{d}\boldsymbol{\omega} \leq C \| \mathbf{u} \|_{L_2(\Omega)}^2.
\end{equation*}
\end{lem}
\begin{proof}
With the definitions of the $\widetilde{\mathbf{E}}_{\scriptstyle \mathrm{div}}$,$E_S$ and $T$ operators, we have
\begin{eqnarray*}
\int\limits_{\mathbb{R}^d} \frac{\Bigl\| \widehat{\widetilde{\mathbf{E}}_{\scriptstyle \text{div}} \mathbf{u}}(\boldsymbol{\omega}) \Bigr\|_2^2}{\|\boldsymbol{\omega}\|_2^2} \mathrm{d}\boldsymbol{\omega} &=& \int\limits_{\mathbb{R}^d} \frac{\Bigl\| \boldsymbol{\omega} \times  \widehat{E_S T \mathbf{u}}(\boldsymbol{\omega}) \Bigr\|_2^2}{\|\boldsymbol{\omega}\|_2^2} \mathrm{d}\boldsymbol{\omega} \\
&\leq& C\int\limits_{\mathbb{R}^d} \Bigl\| \widehat{E_S T \mathbf{u}}(\boldsymbol{\omega}) \Bigr\|_2^2 \mathrm{d}\boldsymbol{\omega} \\
&=& C \|E_S T \mathbf{u}\|^2_{L_2(\mathbb{R}^d)} \\
&\leq& C \|E_S T \mathbf{u}\|^2_{H^1(\mathbb{R}^d)} \\
&\leq& C \|T \mathbf{u}\|^2_{H^1(\Omega)} \\
&\leq& C \|\mathbf{u}\|^2_{L_2(\Omega)},
\end{eqnarray*}
where we have also used that the $E_S$ and $T$ operators are bounded (Lemma \ref{lemExtensionOperator}).
\end{proof}

The following theorem and corollary are our main results on the convergence of the multiscale symmetric collocation algorithm for solving the Stokes problem.
\begin{thm}
\label{thmStokesSubsErrors}
Assume that $\Omega$ and $\mathbf{f},\mathbf{g}$ satisfy the smoothness assumptions of Theorem \ref{thmStokesL2}  for $d=2,3$. Suppose the kernel $\mathbf{\Phi}$ is chosen such that $\mathcal{N}_{\mathbf{\Phi}}(\mathbb{R}^d) = \widetilde{\mathbf{H}}^{\tau}(\mathbb{R}^d;\mathrm{div}) \times H^{\tau-1}(\mathbb{R}^d)$ with $\tau > 0$ and define the scaled kernels by \eqref{eqnDefnPhiDelta} with scale factor $\delta_j$. Then for Algorithm \ref{AlgSymmStokes} there exists a constant $\alpha_1$ such that
\begin{equation}
\|\mathbf{Ee}_j\|_{\mathcal{N}_{\boldsymbol{\Phi}_{j+1}}(\mathbb{R}^d)} \leq \alpha_1 \|\mathbf{Ee}_{j-1}\|_{\mathcal{N}_{\boldsymbol{\Phi}_{j}}(\mathbb{R}^d)},
\end{equation}
where $\alpha_1$ is a constant independent of the point sets $X_1,X_2,\ldots$ and $\mathbf{Ee}_j$ is the extension operator for $\mathbf{v}$ defined in \eqref{eqnStokesExtension} applied to the error at level $j$ defined in Algorithm \ref{AlgSymmStokes}.
\end{thm}
\begin{proof}
With the notation $\mathbf{Ee_j} = (\widetilde{\mathbf{E}}_{\scriptstyle \text{div}}\mathbf{u}-\mathbf{M}_j\widetilde{\mathbf{E}}_{\scriptstyle \text{div}}\mathbf{u},E_S p - M_j E_{\scriptstyle \text{S}}p)^T = (\widetilde{\mathbf{E}}_{\scriptstyle \text{div}}\mathbf{e}_{\mathbf{u},j},E_{\scriptstyle S}e_{p,j})^T$ and with \eqref{eqnStokesPhiDeltaNorm}, we have
\begin{eqnarray*}
\|\mathbf{Ee}_j\|_{\mathcal{N}_{\mathbf{\Phi}_{j+1}}(\mathbb{R}^d)}^2 &\leq& \bar{C}_1\int\limits_{\mathbb{R}^d} \Bigg[\frac{\Bigl\| \widehat{\widetilde{\mathbf{E}}_{\scriptstyle \text{div}} \mathbf{e}_{\mathbf{u},j}}(\boldsymbol{\omega}) \Bigr\|_2^2}{\|\boldsymbol{\omega}\|_2^2} \left(1+\delta_{j+1}^2 \|\boldsymbol{\omega}\|_2^2\right)^{\tau+1} \Bigg. \\
 && \Bigg. +   \quad \left\vert \widehat{E_S e_{p,j}}(\boldsymbol{\omega}) \right\vert^2 \left(1+\delta_{j+1}^2 \|\boldsymbol{\omega}\|_2^2\right)^{\tau-1} \Bigg] \mathrm{d}\boldsymbol{\omega} \nonumber \\
&=:& I_1 + I_2,
\end{eqnarray*}
with
\begin{eqnarray*}
I_1 &:=& \int\limits_{\|\boldsymbol{\omega}\|_2 \leq \frac{1}{\delta_{j+1}}} \Bigg[\frac{\Bigl\| \widehat{\widetilde{\mathbf{E}}_{\scriptstyle \text{div}} \mathbf{e}_{\mathbf{u},j}}(\boldsymbol{\omega}) \Bigr\|_2^2}{\|\boldsymbol{\omega}\|_2^2} \left(1+\delta_{j+1}^2 \|\boldsymbol{\omega}\|_2^2\right)^{\tau+1} + \widehat{E_S e_{p,j}}(\boldsymbol{\omega}) \left(1+\delta_{j+1}^2 \|\boldsymbol{\omega}\|_2^2\right)^{\tau-1} \Bigg] \mathrm{d}\boldsymbol{\omega},  \\
I_2 &:=& \int\limits_{\|\boldsymbol{\omega}\|_2 \geq \frac{1}{\delta_{j+1}}} \Bigg[\frac{\Bigl\| \widehat{\widetilde{\mathbf{E}}_{\scriptstyle \text{div}} \mathbf{e}_{\mathbf{u},j}}(\boldsymbol{\omega}) \Bigr\|_2^2}{\|\boldsymbol{\omega}\|_2^2} \left(1+\delta_{j+1}^2 \|\boldsymbol{\omega}\|_2^2\right)^{\tau+1} + \widehat{E_S e_{p,j}}(\boldsymbol{\omega}) \left(1+\delta_{j+1}^2 \|\boldsymbol{\omega}\|_2^2\right)^{\tau-1} \Bigg] \mathrm{d}\boldsymbol{\omega}.  \\
\end{eqnarray*}
For $I_1$, we can use that $\delta_{j+1} \|\boldsymbol{\omega}\|_2 \leq 1$, Lemma \ref{lemTechErrBoundU}, Theorem \ref{thmStokesL2} and Lemma \ref{lemStokesNormEquiv} to yield
\begin{eqnarray*}
I_1 &\leq& C \left( \widetilde{\mathbf{E}}_{\scriptstyle \text{div}} \|\mathbf{e}_{\mathbf{u},j}\|^2_{\mathbf{L}_2(\mathbb{R}^d)} + \|E_S e_{p,j}\|^2_{L_2(\mathbb{R}^d)} \right) \\
&\leq& C \left( \|\mathbf{e}_{\mathbf{u},j}\|^2_{\mathbf{L}_2(\Omega)} + \|e_{p,j}\|^2_{L_2(\Omega)} \right) \\
&\leq& C \bar{h}_j^{2\tau-4} \|\mathbf{Ee}_j\|_{\mathcal{N}_{\mathbf{\Phi}}(\mathbb{R}^d)}^2 \\
&\leq& C \frac{\bar{h}_j^{2\tau-4}}{\delta_j^{2\tau+2}} \|\mathbf{Ee}_{j-1}\|_{\mathcal{N}_{\mathbf{\Phi}_j}(\mathbb{R}^d)}^2 \\
&=& C_1 \beta^{-2\tau-2} \|\mathbf{Ee}_{j-1}\|_{\mathcal{N}_{\mathbf{\Phi}_j}(\mathbb{R}^d)}^2,
\end{eqnarray*}
where in the second last step we have used that since the interpolant at $X_j$  to $\mathbf{e}_{j-1}$ is the same as
the interpolant to $\mathbf{E e}_{j-1}$ (both functions take the same values
on $X_j \subseteq \Omega$), we have
\begin{align*}
 \|\mathbf{e}_j\|_{H^\tau(\Omega)} &= \|\mathbf{e}_{j-1} - \mathbf{S}_{j} \mathbf{e}_{j-1}\|_{\mathbf{H}^\tau(\Omega)} \\
   &= \| \mathbf{E e}_{j-1} - \mathbf{S}_{j} \mathbf{E e}_{j-1}\|_{\mathbf{H}^\tau(\Omega)} \\
   &\le \| \mathbf{E e}_{j-1} - \mathbf{S}_{j} \mathbf{E e}_{j-1}\|_{\mathbf{H}^\tau(\mathbb{R}^d)} \\
   &\le C\delta_j^{-\tau-1} \| \mathbf{E e}_{j-1} - \mathbf{S}_{j} \mathbf{E e}_{j-1}\|_{\mathcal{N}_{\mathbf{\Phi}_j}(\mathbb{R}^d)}  \\
   &\le C \delta_j^{-\tau-1} \| \mathbf{E e}_{j-1} \|_{\mathcal{N}_{\mathbf{\Phi}_j}(\mathbb{R}^d)}.
\end{align*}

For $I_2$, since $\delta_{j+1} \|\boldsymbol{\omega}\|_2 \geq 1$, we have
\[
\left(1 + \delta^2_{j+1} \|\boldsymbol{\omega}\|_2^2 \right)^{\tau-1}  \leq \left(2 \delta_{j+1}^2 \|\boldsymbol{\omega}\|_2^2 \right)^{\tau-1} \leq 2^{\tau-1} \mu^{\tau-3} \left(1 + \delta^2_{j} \|\boldsymbol{\omega}\|_2^2 \right)^{\tau-1},
\]
since if $\mu,\delta \leq 1$, we have
\begin{equation*}
\delta_{j+1}^{2\tau-2} \leq \mu^{\frac{(2\tau-2)(\tau-3)}{\tau}} \delta_j^{2\tau-2} \leq \mu^{\tau-3} \delta_j^{2\tau-2}.
\end{equation*}
Similarly, we have
\[
\left(1 + \delta^2_{j+1} \|\boldsymbol{\omega}\|_2^2 \right)^{\tau+1}  \leq 2^{\tau+1} \mu^{\tau-1} \left(1 + \delta^2_{j} \|\boldsymbol{\omega}\|_2^2 \right)^{\tau+1}.
\]
Hence
\begin{eqnarray*}
I_2 &\leq& C \mu^{\tau-3} \|\mathbf{Ee}_j\|_{\mathcal{N}_{\mathbf{\Phi}_j}(\mathbb{R}^d)}^2 \\
&\leq& C_2 \mu^{\tau-3} \|\mathbf{Ee}_{j-1}\|^2_{\mathcal{N}_{\mathbf{\Phi}_j}(\mathbb{R}^d)},
\end{eqnarray*}
where the last step follows in the same way as the last part of the derivation for $I_1$. The result follows with
\[
\alpha_1 := \left(C_1 \beta^{-2\tau-2} + C_2 \mu^{\tau-3} \right)^{1/2}.
\]
\end{proof}
\begin{cor}
There exist positive constants $C_3$ and $C_4$ such that
\[
\|\mathbf{v}-\mathbf{M}_n \mathbf{v}\|_{\mathbf{L}_2(\Omega)} \leq C_3 \alpha_1^n \left( \|\mathbf{u}\|_{\mathbf{H}^{\tau}(\Omega)} + \|p\|_{H^{\tau-1}(\Omega)} \right) \quad \mbox{ for } \,\, n=1,2,\ldots
\]
and
\[
\|\mathbf{u}-\mathbf{M}_n \mathbf{u}\|_{\mathbf{L}_2(\partial \Omega)} \leq C_4 \alpha_1^n \left( \|\mathbf{u}\|_{\mathbf{H}^{\tau}(\Omega)} + \|p\|_{H^{\tau-1}(\Omega)} \right) \quad \mbox{ for } \,\, n=1,2,\ldots
\]
Thus the multiscale approximation $\mathbf{M}_n \mathbf{v} $ resulting from Algorithm \ref{AlgSymmStokes} converges linearly to $\mathbf{v}$ in the $L_2-$norm in $\Omega$ and on $\partial \Omega$ if $\alpha_1 < 1$.
\end{cor}
\begin{proof}
With Lemma \ref{lemStokesNormEquiv} and Theorems \ref{thmStokesL2} and \ref{thmStokesSubsErrors} we have
\begin{eqnarray*}
\|\mathbf{v}-\mathbf{M}_n \mathbf{v}\|_{\mathbf{L}_2(\Omega)} &=& \|\mathbf{e}_n \|_{\mathbf{L}_2(\Omega)} \\
&\leq& C \bar{h}_{1,n}^{\tau -2} \|\mathbf{Ee}_n \|_{\mathcal{N}_{\mathbf{\Phi}}(\mathbb{R}^d)} \\
&\leq& C  \|\mathbf{Ee}_n \|_{\mathcal{N}_{\mathbf{\Phi}_{n+1}}(\mathbb{R}^d)} \\
&\leq& C \alpha_1^n \|\mathbf{Ev}\|_{\mathcal{N}_{\mathbf{\Phi}_{1}}(\mathbb{R}^d)} \\
&\leq& C \alpha_1^n \|\mathbf{Ev}\|_{\mathcal{N}_{\mathbf{\Phi}}(\mathbb{R}^d)} \\
&\leq& C \alpha_1^n \left(\|\mathbf{u}\|_{\mathbf{H}^{\tau}(\Omega)} + \|p\|_{H^{\tau-1}(\Omega)} \right),
\end{eqnarray*}
which proves the first result. For the second result, with \eqref{eqnVelBoundaryError} we can see that
\begin{eqnarray*}
\|\mathbf{u}-\mathbf{M}_n \mathbf{u}\|_{\mathbf{L}_2(\partial \Omega)} &\leq& \|\mathbf{u}-\mathbf{M}_n \mathbf{u}\|_{\mathbf{H}^{3/2}(\partial \Omega)} \\
&\leq& C h_{2,n}^{\tau-2} \|\mathbf{u}-\mathbf{M}_n \mathbf{u}\|_{\mathbf{H}^{\tau}(\Omega)} \\
&\leq& C \bar{h}_{n}^{\tau -2} \|\mathbf{Ee}_n \|_{\mathcal{N}_{\mathbf{\Phi}}(\mathbb{R}^d)},
\end{eqnarray*}
and the remainder of the proof is the same as for the first result.
\end{proof}

\section{Condition number} \label{StokesCondNumber}

In this section, we present upper and lower bounds for the eigenvalues of the multiscale symmetric collocation algorithm for the Stokes problem.  At each step of the multiscale algorithm, we need to solve a linear system resulting from the collocation conditions \eqref{eqnColl1} and \eqref{eqnColl2} on a set $X = \{\mathbf{x}_1,\ldots,\mathbf{x}_M\}$:
\[
\mathbf{A}_\delta \mathbf{b} = (\mathbf{f}\,\,\, \mathbf{g})^T.
\]
Since the collocation matrix $\mathbf{A}_{\delta}$ is symmetric and positive definite, we know that the condition number is given by
\begin{equation}
\kappa(\mathbf{A}_{\delta}) = \frac{\lambda_{\max}(\mathbf{A}_{\delta})}{\lambda_{\min}(\mathbf{A}_{\delta})} \label{eqnCondNumberMaxMinEig},
\end{equation}
where $\lambda_{\max}(\mathbf{A}_{\delta})$ and $\lambda_{\min}(\mathbf{A}_{\delta})$ denote the maximum and minimum eigenvalues of $\mathbf{A}_{\delta}$.

We will first need several technical lemmas concerning derivatives of the Wendland functions.
\begin{lem}
\label{lemTech1}
With spatial dimension $d$ and smoothness parameter $k = 2, 3,\ldots$ let $\psi_{\ell,k}$ be the original Wendland function. Then with $\mathbf{x},\mathbf{y} \in \mathbb{R}^d$ and $1 \leq i,j \leq d$ and $i \neq j$, we have
\[
\partial_{ij} \Psi_{\ell,k}(\mathbf{x}-\mathbf{y})\vert_{\mathbf{x}=\mathbf{y}} = 0.
\]
\end{lem}
\begin{proof}
We recall that the Wendland functions are piecewise polynomials with support $[0,1]$ and we can write \cite{Wen05}
\begin{equation}
\label{eqnWF}
\psi_{\ell,k}(r) = \sum_{i=0}^{2k+\ell} b_i\, r^i, \quad r \in [0,1]
\end{equation}
and that the first $k$ odd coefficients $\{b_{2i+1}\}_{i=0}^k$ vanish. With the chain rule, where $\mathbf{x} - \mathbf{y} = (x_1 - y_1,\ldots, x_d - y_d)$ and $r = \|\mathbf{x}-\mathbf{y}\|_2$, we have
\[
\partial_{ij} \Psi_{\ell,k}(\mathbf{x} - \mathbf{y}) = \frac{(x_i-y_i)(x_j-y_j)}{r^2} \left( \psi_{\ell,k}^{(2)}(r) - \frac{1}{r}\psi_{\ell,k}^{(1)}(r)\right).
\]
Using \eqref{eqnWF}, this last expression becomes
\begin{eqnarray}
\partial_{ij} \Psi_{\ell,k}(\mathbf{x} - \mathbf{y}) &=& \frac{(x_i-y_i)(x_j-y_j)}{r^2} \left( \sum_{i=2}^{2k+\ell} b_i \, (i-1)_2\, r^{i-2} - \sum_{i=1}^{2k+\ell} i \, b_i \, r^{i-2}\right), \nonumber \\
&=:& \frac{(x_i-y_i)(x_j-y_j)}{r^2} \left( \sum_{i=1}^{2k+\ell} \bar{b}_i \, r^{i-2} \right)
\end{eqnarray}
where
\[
(c)_n := \frac{\Gamma(c+n)}{\Gamma(c)}
\]
denotes the Pochhammer symbol. Now the first three coefficients $\{\bar{b}_i\}_{i=1}^3$ are
\begin{eqnarray*}
\bar{b}_1 &=& b_1 = 0 \\
\bar{b}_2 &=& (2 - 2)b_2 = 0 \\
\bar{b}_3 &=& 0,
\end{eqnarray*}
since the first $k$ odd coefficients of the Wendland polynomial are zero and $k\geq2$. Hence we can write
\[
\partial_{ij} \Psi_{\ell,k}(\mathbf{x} - \mathbf{y}) =
(x_i-y_i)(x_j-y_j) \left( \sum_{i=4}^{2k+\ell} \bar{b}_i \, r^{i-4} \right),
\]
and the result follows immediately.
\end{proof}

\begin{lem}
\label{lemTech2}
With spatial dimension $d$ and smoothness parameter $k = 3,4,\ldots$ let $\psi_{\ell,k}$ be the original Wendland function. Then with $\mathbf{x},\mathbf{y} \in \mathbb{R}^d$ and $1 \leq i,j \leq d$ and $i \neq j$, we have
\[
\partial_{ij} \Delta^2 \Psi_{\ell,k}(\mathbf{x}-\mathbf{y})\vert_{\mathbf{x}=\mathbf{y}} = 0.
\]
\end{lem}
\begin{proof}
Once again employing the chain rule gives
\begin{multline*}
\partial_{ij} \Delta^2 \Psi_{\ell,k}(\mathbf{x}-\mathbf{y}) = \frac{(x_i-y_i)(x_j-y_j)}{r^2} \\ \times \left( \psi_{\ell,k}^{(6)}(r) + \frac{1}{r}\psi_{\ell,k}^{(5)}(r) - \frac{7}{r^2}\psi_{\ell,k}^{(4)}(r) + \frac{12}{r^3}\psi_{\ell,k}^{(3)}(r) - \frac{15}{r^4}\psi_{\ell,k}^{(2)}(r) + \frac{15}{r^5}\psi_{\ell,k}^{(1)}(r)\right).
\end{multline*}
With \eqref{eqnWF} we can rewrite this as
\begin{eqnarray*}
&& \partial_{ij} \Delta^2 \Psi_{\ell,k}(\mathbf{x}-\mathbf{y}) = \frac{(x_i-y_i)(x_j-y_j)}{r^2} \Bigg[ \sum_{i=6}^{2k+\ell} b_i (i-5)_6 r^{i-6} + \sum_{i=5}^{2k+\ell} b_i (i-4)_5 r^{i-6} \\&& - 7\sum_{i=4}^{2k+\ell} b_i (i-3)_4 r^{i-6} + 12\sum_{i=3}^{2k+\ell} b_i (i-2)_3 r^{i-6} - 15\sum_{i=2}^{2k+\ell} b_i (i-1)_2 r^{i-6} + 15\sum_{i=1}^{2k+\ell} i \, b_i r^{i-6} \Bigg] \\
&& \quad \quad \quad \quad \quad \quad \quad \quad =: \sum_{i=1}^{2k+\ell} \tilde{b}_i r^{i-6}.
\end{eqnarray*}
Since $\tilde{b}_i = C(i) b_i$, the first $k$ odd coefficients $\{\tilde{b}_{2i+1}\}_{i=0}^k$ are zero. Then we can determine other coefficients as
\begin{eqnarray*}
\tilde{b}_2 &=& 30(b_2 - b_2) = 0 \\
\tilde{b}_4 &=& b_4(60 - 15(3)_2 + 12(2)_3 - 7(1)_4) = 0 \\
\tilde{b}_6 &=& b_6(90 - 15(5)_2 + 12(4)_3 - 7(3)_4 + (1)_6) = 0.
\end{eqnarray*}
Hence since $k = 3, 4, \ldots$, we can write
\[
\partial_{ij} \Delta^2 \Psi_{\ell,k}(\mathbf{x}-\mathbf{y}) = (x_i-y_i)(x_j-y_j) \sum_{i=8}^{2k+\ell} \tilde{b}_i r^{i-8},
\]
and the result follows immediately.
\end{proof}

\begin{lem}
\label{lemTech3}
With spatial dimension $d$ and smoothness parameter $k = 2, 3,\ldots$ let $\psi_{\ell,k}$ be the original Wendland function. Then with $\mathbf{x},\mathbf{y} \in \mathbb{R}^d$ and $1 \leq j \leq d$ we have
\[
\partial_{jj} \Psi_{\ell,k}(\mathbf{x}-\mathbf{y})\vert_{\mathbf{x}=\mathbf{y}} < 0,
\]
and is independent of $j$.
\end{lem}
\begin{proof}
With the chain rule, where once again $\mathbf{x} - \mathbf{y} = (x_1 - y_1,\ldots, x_d - y_d)$ and $r = \|\mathbf{x}-\mathbf{y}\|_2$, we have
\[
\partial_{jj} \Psi_{\ell,k}(\mathbf{x} - \mathbf{y}) = \frac{(x_j-y_j)^2}{r^2} \left( \psi_{\ell,k}^{(2)}(r) - \frac{1}{r}\psi_{\ell,k}^{(1)}(r)\right) + \frac{1}{r}\psi_{\ell,k}^{(1)}(r).
\]
With Lemma \ref{lemTech1}, the term in brackets is equal to zero when $\mathbf{x} = \mathbf{y}$. Using \eqref{eqnWF} and noting that the first $k$ odd coefficients are zero, this last term becomes
\begin{equation*}
\frac{1}{r}\psi_{\ell,k}^{(1)}(r) = \sum_{i=2}^{2\ell+k} i \, b_i \, r^{i-2},
\end{equation*}
which means that the case of $\mathbf{x}=\mathbf{y}$, which is equivalent to $r=0$, reduces down to $2b_2$. Now combining positive terms into a generic constant $C$, we have from \cite{CheH12}
\begin{eqnarray*}
b_2 &=& C (-1)^k {\frac{1}{2} \choose k} \\
&=& C \frac{(-1)^k}{\Gamma(\frac{1}{2}-(k-1))} \\
&=& C (-1)^k (-1)^{k-1} < 0,
\end{eqnarray*}
where we have also used \cite[8.339.3]{GraR07}
\[
\Gamma\left( \frac{1}{2}-n \right) = \sqrt{\pi} \frac{(-4)^n n!}{(2n)!}.
\]
\end{proof}

\begin{lem}
\label{lemTech4}
With spatial dimension $d$ and smoothness parameter $k = 3,4,\ldots$ let $\psi_{\ell,k}$ be the original Wendland function. Then with $\mathbf{x},\mathbf{y} \in \mathbb{R}^d$ and $1 \leq j \leq d$ we have
\[
\partial_{jj} \Delta^2 \Psi_{\ell,k}(\mathbf{x}-\mathbf{y})\vert_{\mathbf{x}=\mathbf{y}} < 0,
\]
and is independent of $j$.
\end{lem}
\begin{proof}
With the chain rule, where once again $\mathbf{x} - \mathbf{y} = (x_1 - y_1,\ldots, x_d - y_d)$ and $r = \|\mathbf{x}-\mathbf{y}\|_2$, we have
\begin{multline*}
\partial_{jj} \Delta^2 \Psi_{\ell,k}(\mathbf{x} - \mathbf{y}) = \frac{(x_j-y_j)^2}{r^2} \\ \times \left( \psi_{\ell,k}^{(6)}(r) + \frac{1}{r}\psi_{\ell,k}^{(5)}(r) - \frac{7}{r^2}\psi_{\ell,k}^{(4)}(r) + \frac{12}{r^3}\psi_{\ell,k}^{(3)}(r) - \frac{15}{r^4}\psi_{\ell,k}^{(2)}(r) + \frac{15}{r^5}\psi_{\ell,k}^{(1)}(r)\right) \\
+ \frac{1}{r}\left( \psi_{\ell,k}^{(5)}(r) + \frac{2}{r}\psi_{\ell,k}^{(4)}(r) - \frac{3}{r^2}\psi_{\ell,k}^{(3)}(r) + \frac{3}{r^3}\psi_{\ell,k}^{(2)}(r) - \frac{3}{r^4}\psi_{\ell,k}^{(1)}(r) \right).
\end{multline*}
With Lemma \ref{lemTech1}, the first term in the previous expression is equal to zero when $\mathbf{x} = \mathbf{y}$. As before, we can write the second term as a series
\[
\sum_{i=1}^{2k+\ell} \underbar{b}_i \, r^{i-6}.
\]
Since $k \geq 3$, $\underbar{b}_1= \underbar{b}_3 = \underbar{b}_5 = 0$ and equating coefficients gives
\begin{eqnarray*}
\underbar{b}_2 &=& b_2(-6+3(1)_2) = 0 \\
\underbar{b}_4 &=& b_4(-12+3(3)_2 - 3(2)_3 + 2(1)_4) = 0,
\end{eqnarray*}
which means we are left with
\[
\sum_{i=6}^{2k+\ell} \underbar{b}_i \, r^{i-6}.
\]
Hence the case of $\mathbf{x}=\mathbf{y}$, which is equivalent to $r=0$, reduces down to $\underbar{b}_6$ which is given by
\[
\underbar{b}_6 = ((2)_5 + 2(3)_4 - 3(4)_3 + 3(5)_2 - 18)b_6 = 1152b_6.
\]
As before, combining positive terms into a generic constant $C$ and noting that $k =3,4,\ldots$, we have from \cite{CheH12}
\begin{eqnarray*}
b_6 &=& C (-1)^k {\frac{5}{2} \choose k} \\
&=& C \frac{(-1)^k}{\Gamma(\frac{1}{2}-(k-3))} \\
&=& C (-1)^k (-1)^{k-3} < 0.
\end{eqnarray*}
\end{proof}

The next theorem gives a lower bound on the minimum eigenvalue of $\mathbf{A}_{\delta}$.

\begin{thm}
\label{thmMinEig1}
Suppose the kernel $\mathbf{\Phi}$ is defined by \eqref{eqnDefnPhi} and define the scaled kernel $\mathbf{\Phi}_{\delta}$ by \eqref{eqnDefnPhiDelta} with a positive scaling factor $\delta$. Then the smallest eigenvalue of the collocation matrix defined by \eqref{eqnColl1} and \eqref{eqnColl2} can be bounded by
\[
\lambda_{\min}(\mathbf{A}) \geq C \left( \frac{q_X}{\delta}\right)^{2\tau+2} q_X^{-d-2},
\]
where the constant $C$ is independent of the pointset $X$.
\end{thm}
\begin{proof}
We follow the proof of \cite[Theorem 4.1]{GieW06}. We will adopt the functional notation
\begin{equation*}
\xi_{i,j}(\mathbf{v}) = \left\{
   \begin{array}{rl}
   \left(L \mathbf{v} \right)_i (\mathbf{x}_j)\, & \text{ for } \quad 1\leq j \leq N, \quad 1\leq i \leq d, \\
   \mathbf{v}_i (\mathbf{x}_j)\, & \text{ for } \quad N+1\leq j \leq M, \quad 1\leq i \leq d.
   \end{array} \right.
\end{equation*}
We will use the superscript $\mathbf{y}$ to denote that the functional acts with respect to its second argument. Then with $\beta \in \mathbb{R}^{dM}$, we need to show that
\begin{equation}
\label{eqnMinStep1}
\sum_{i,i'=1}^d \sum_{j,k=1}^M \beta_{i,j} \beta_{i',k} \xi_{i,j} \xi_{i',k}^{\mathbf{y}} \mathbf{\Phi}_{\delta}(\mathbf{x}-\mathbf{y}) \geq C \left( \frac{q_X}{\delta}\right)^{2\tau+2} q_X^{-d-2} \|\beta\|_2^2.
\end{equation}
Now with the inverse Fourier transform, the left hand side of \eqref{eqnMinStep1} becomes
\begin{multline*}
\sum_{i,i'=1}^d \sum_{j,k=1}^M \beta_{i,j} \beta_{i',k} \xi_{i,j} \xi_{i',k}^{\mathbf{y}} \mathbf{\Phi}_{\delta}(\mathbf{x}-\mathbf{y}) \\ = (2\pi)^{-d/2} \int_{\mathbb{R}^d} \sum_{i,i'=1}^d \sum_{j,k=1}^M \beta_{i,j} \beta_{i',k} \xi_{i,j} \xi_{i',k}^{\mathbf{y}} \widehat{\mathbf{\Phi}}_{\delta}(\boldsymbol{\omega}) e^{I(\mathbf{x}-\mathbf{y})\cdot \boldsymbol{\omega}}, \mathrm{d} \boldsymbol{\omega} \\
\end{multline*}
where $I^2 = -1$.
Now we define a second scaled kernel $\mathbf{\Phi}_a$ by \eqref{eqnDefnScaledKernel} with $0 < a \leq 1$ and $a \leq \delta$. For $\delta \leq 1$ we have
\begin{equation}
\left(1+\delta^2 \|\boldsymbol{\omega}\|_2^2\right)^{\tau-1} \geq \left(\delta^2+\delta^2 \|\boldsymbol{\omega}\|_2^2\right)^{\tau-1} \geq \delta^{2\tau-2} \left(1+\|\boldsymbol{\omega}\|_2^2\right)^{\tau-1}, \label{eqnStokesNormEquivStep1}
\end{equation}
and recalling that $\psi_{\tau-1}$ satisfies \eqref{eqnFourierPsi} gives
\begin{eqnarray*}
\widehat{\psi_{\tau-1,\delta}}(\boldsymbol{\omega}) &=& \widehat{\psi_{\tau-1}}(\delta \boldsymbol{\omega}) \geq c_{1,\tau-1} \left(1 + \|\delta \boldsymbol{\omega}\|_2^2\right)^{-\tau+1} \\
&=& c_{1,\tau-1} \left( \frac{a}{\delta} \right)^{2\tau -2} \left(\left(\frac{a}{\delta}\right)^2 + \|a \boldsymbol{\omega}\|_2^2\right)^{-\tau+1} \\
&\geq& c_{1,\tau-1} \left( \frac{a}{\delta} \right)^{2\tau -2} \left(1 + \|a \boldsymbol{\omega}\|_2^2\right)^{-\tau+1} \\
&\geq& \frac{c_{1,\tau-1}}{c_{2,\tau-1}} \left( \frac{a}{\delta} \right)^{2\tau -2} \widehat{\psi_{\tau-1,a}}(\boldsymbol{\omega}).
\end{eqnarray*}
Since $\psi_{\tau+1}$ satisfies \eqref{eqnFourierUpsilon} and with \eqref{eqnFourierUpsilonDefn}, we proceed similarly to get
\begin{eqnarray*}
\widehat{\mathbf{\Psi}_{\tau+1,\delta}}(\boldsymbol{\omega}) &=& = \left(\|\boldsymbol{\omega}\|_2^2 \mathbf{I} - \boldsymbol{\omega} \boldsymbol{\omega}^T\right) \widehat{\psi_{\tau+1}}(\delta \|\boldsymbol{\omega}\|_2) \\
&=& c_{1,\tau+1} \left( \frac{a}{\delta} \right)^{2\tau+2} \left(\|\boldsymbol{\omega}\|_2^2 \mathbf{I} - \boldsymbol{\omega} \boldsymbol{\omega}^T\right) \left( \left(\frac{a}{\delta}\right)^2 + \| a \boldsymbol{\omega}\|_2^2\right)^{-\tau-1} \\
&\geq& c_{1,\tau+1}\left( \frac{a}{\delta} \right)^{2\tau+2} \left(\|\boldsymbol{\omega}\|_2^2 \mathbf{I} - \boldsymbol{\omega} \boldsymbol{\omega}^T\right) \left( 1 + \| a \boldsymbol{\omega}\|_2^2\right)^{-\tau-1} \\
&\geq& \frac{c_{1,\tau+1}}{c_{2,\tau+1}} \left( \frac{a}{\delta} \right)^{2\tau+2} \left(\|\boldsymbol{\omega}\|_2^2 \mathbf{I} - \boldsymbol{\omega} \boldsymbol{\omega}^T\right) \widehat{\psi_{\tau+1}}(a \|\boldsymbol{\omega}\|_2) \\
&=& \frac{c_{1,\tau+1}}{c_{2,\tau+1}}\left( \frac{a}{\delta} \right)^{2\tau+2} \widehat{\mathbf{\Psi}_{\tau+1,a}}(\boldsymbol{\omega}).
\end{eqnarray*}
Since $a/\delta < 1$, we have the following bound on $\widehat{\mathbf{\Phi}_{\delta}}$
\[
\widehat{\mathbf{\Phi}_{\delta}}(\boldsymbol{\omega}) \geq c \left(\frac{a}{\delta}\right)^{2\tau+2} \widehat{\mathbf{\Phi}_{a}}(\boldsymbol{\omega}),
\]
and hence we have
\[
\sum_{i,i'=1}^d \sum_{j,k=1}^M \beta_{i,j} \beta_{i',k} \xi_{i,j} \xi_{i',k}^{\mathbf{y}} \mathbf{\Phi}_{\delta}(\mathbf{x}-\mathbf{y}) \geq c \left(\frac{a}{\delta}\right)^{2\tau+2} \sum_{i,i'=1}^d \sum_{j,k=1}^M \beta_{i,j} \beta_{i',k} \xi_{i,j} \xi_{i',k}^{\mathbf{y}} \mathbf{\Phi}_{a}(\mathbf{x}-\mathbf{y}),
\]
and if we select $a=q_X \leq 1$ such that we need only consider entries of the quadratic form corresponding to equal centres, with the definition of the scaled kernel in \eqref{eqnDefnScaledKernel}, this reduces to
\begin{multline*}
\sum_{i,i'=1}^d \sum_{j,k=1}^M \beta_{i,j} \beta_{i',k} \xi_{i,j} \xi_{i',k}^{\mathbf{y}} \mathbf{\Phi}_{\delta}(\mathbf{x}-\mathbf{y}) \\ \geq c \left(\frac{q_X}{\delta}\right)^{2\tau+2} q_X^{-d} \sum_{i=1}^d \left\{ \sum_{j=1}^N \beta_{i,j}^2 \left( -\sum_{j=\{1:d\}\setminus i} q_X^{-6}\partial_{jj} \Delta^2 \psi_{\tau+1}(0) - q_X^{-2} \partial_{ii} \psi_{\tau-1}(0) \right) + \right. \\
\left. \quad \quad \sum_{j=N+1}^M \beta_{i,j}^2 \left( -\sum_{j=\{1:d\}\setminus i} q_X^{-2} \partial_{jj} \psi_{\tau-1}(0) \right) \right\},
 \end{multline*}
 since for interior centres we have
\begin{equation}
\label{eqnIntDiagVals}
\xi_{i,j} \xi_{i',k}^{\mathbf{y}} \mathbf{\Phi}(\mathbf{x}-\mathbf{y})\vert_{j=k} = \left\{
   \begin{array}{rl}
   -\nu^2 \sum_{j=1:d\setminus i} \partial_{jj} \Delta^2 \psi_{\tau+1}(0) - \partial_{ii}\psi_{\tau-1}(0) \, & \text{ for } \quad i = i', \\
    \nu^2 \partial_{i i'} \Delta^2 \psi_{\tau+1}(0) - \partial_{i i'} \psi_{\tau-1}(0) = 0 & \text{ for } \quad i \neq i',
   \end{array} \right.
\end{equation}
with Lemmas \ref{lemTech1} and \ref{lemTech2}. Similarly for the boundary centres
\begin{equation}
\label{eqnBdyDiagVals}
\xi_{i,j} \xi_{i',k}^{\mathbf{y}} \mathbf{\Phi}(\mathbf{x}-\mathbf{y})\vert_{j=k} = \left\{
   \begin{array}{rl}
   -\sum_{j=1:d\setminus i} \partial_{jj} \psi_{\tau-1}(0) \, & \text{ for } \quad i = i', \\
    - \partial_{i i'} \psi_{\tau-1}(0) =0 & \text{ for } \quad i \neq i',
   \end{array} \right.
\end{equation}
and then the result follows as
\begin{equation*}
\sum_{i=1}^d \sum_{j,k=1}^M \beta_{i,j} \beta_{i,k} \xi_{i,j} \xi_{i,k}^{\mathbf{y}} \mathbf{\Phi}_{\delta}(\mathbf{x}-\mathbf{y}) \geq c \, \tilde{c} \, \left( \frac{q_X}{\delta}\right)^{2\tau+2} q_X^{-d-2} \|\beta\|_2^2,
\end{equation*}
with Lemmas \ref{lemTech3} and \ref{lemTech4} which give
\begin{eqnarray*}
\tilde{c} &:=& \min_{1\leq i \leq d} \left(-\sum_{j=\{1:d\}\setminus i} q_X^{-4}\partial_{jj} \Delta^2 \psi_{\tau+1}(0) - \partial_{ii} \psi_{\tau-1}(0), -\sum_{j=\{1:d\}\setminus i} \partial_{jj} \psi_{\tau-1}(0)\right) \\
&\geq& \min_{1\leq i \leq d} \left(-\sum_{j=\{1:d\}\setminus i} \partial_{jj} \Delta^2 \psi_{\tau+1}(0) - \partial_{ii} \psi_{\tau-1}(0), -\sum_{j=\{1:d\}\setminus i} \partial_{jj} \psi_{\tau-1}(0)\right) \\
&=& \min \left(-\sum_{j=2}^d \partial_{jj} \Delta^2 \psi_{\tau+1}(0) - \partial_{11} \psi_{\tau-1}(0), -\sum_{j=2}^d \partial_{jj} \psi_{\tau-1}(0)\right),
\end{eqnarray*}
since $\psi_{\tau+1}$ is a radial function and $\partial_{ii}\psi_{\tau-1}(0)$ is independent of $i$ from Lemma \ref{lemTech3}.
\end{proof}

Our next result bounds the maximum eigenvalue $\lambda_{\max}(\mathbf{A}_{\delta})$.

\begin{thm}
\label{thmMaxEig}
Suppose the kernel $\mathbf{\Phi}_{\delta}$ is defined as in Theorem \ref{thmMinEig1}. Then if we assume that
\begin{equation}
M \leq C \bar{h}^{-d}, \label{eqnMaxEigUniformness}
\end{equation}
where $M$ denotes the number of (interior and boundary) centres, then the largest eigenvalue of the collocation matrix constructed with $\mathbf{\Phi}_{\delta}$ defined by \eqref{eqnColl1} and \eqref{eqnColl2} can be bounded by
\[
\lambda_{\max}(\mathbf{A}_{\delta}) \leq C \, \delta^{-d-2} \, \bar{h}^{-d},
\]
if $\delta \geq 1$ and by
\[
\lambda_{\max}(\mathbf{A}_{\delta}) \leq C \, \delta^{-d-6} \, \bar{h}^{-d},
\]
if $\delta < 1$, where the constants $C$ are independent of the pointset $X$.
\end{thm}
\begin{proof}
Using the notation from Theorem \ref{thmMinEig1}, together with Gershgorin's theorem, we have
\[
\vert \lambda_{\max}(\mathbf{A}_{\delta}) - \xi_{i,j} \xi_{i,j}^{\mathbf{y}} \mathbf{\Phi}_{\delta}(\mathbf{x},\mathbf{x}) \vert \leq \sum^d_{i'=1} \sum^M_{\substack{k=1 \\ i' \neq i, k \neq j}} \vert  \xi_{i,j} \xi_{i,k}^{\mathbf{y}} \mathbf{\Phi}_{\delta}(\mathbf{x},\mathbf{y}) \vert, \quad 1 \leq i \leq d
\]
which since $\mathbf{\Phi}$ is positive definite, using \eqref{eqnMaxEigUniformness}, Lemmas \ref{lemTech3} and \ref{lemTech4}, the definition of the scaled kernels \eqref{eqnDefnScaledKernel} and \eqref{eqnIntDiagVals} and \eqref{eqnBdyDiagVals}, if $\delta \geq 1$
\begin{eqnarray*}
\lambda_{\max}(\mathbf{A}_{\delta}) &\leq& dM \| \xi_{i',\cdot} \xi_{i,\cdot}^{\mathbf{y}} \mathbf{\Phi}_{\delta}(\cdot,\cdot)\|_{L_{\infty}(\Omega \times \Omega)} \\
&\leq& C \, d \, \bar{h}^{-d} \max \left(-\sum_{j=2}^d \partial_{jj} \Delta^2 \psi_{\tau+1,\delta}(0) - \partial_{11}\psi_{\tau-1,\delta}(0),-\sum_{j=2}^d \partial_{jj} \psi_{\tau-1,\delta}(0)\right) \\
&\leq& C \, d \, \bar{h}^{-d} \delta^{-d-2} \max\left(-\sum_{j=2}^d \partial_{jj} \Delta^2 \psi_{\tau+1}(0) - \partial_{11}\psi_{\tau-1}(0),-\sum_{j=2}^d \partial_{jj} \psi_{\tau-1}(0)\right),
\end{eqnarray*}
where in the last step we have used that
\[
\partial_{jj} \Delta^2 \psi_{\tau+1,\delta}(0) = \delta^{-d} \partial_{jj} \Delta^2 \delta^{-6} \psi_{\tau+1}(0) \leq  \delta^{-d-2} \partial_{jj} \Delta^2 \psi_{\tau+1}(0),
\]
since $\delta \geq 1$. If $\delta < 1$, we have
\begin{eqnarray*}
\lambda_{\max}(\mathbf{A}_{\delta}) &\leq& dM \| \xi_{i',\cdot} \xi_{i,\cdot}^{\mathbf{y}} \mathbf{\Phi}_{\delta}(\cdot,\cdot)\|_{L_{\infty}(\Omega \times \Omega)} \\
&\leq& C \, d \, \bar{h}^{-d} \delta^{-d-6} \max\left(-\sum_{j=2}^d \partial_{jj} \Delta^2 \psi_{\tau+1}(0) - \partial_{11}\psi_{\tau-1}(0),-\sum_{j=2}^d \partial_{jj} \psi_{\tau-1}(0)\right),
\end{eqnarray*}
where in the last step we have used, for example, that
\[
\partial_{11} \psi_{\tau-1,\delta}(0) = \delta^{-d} \partial_{11} \delta^{-2} \psi_{\tau-1}(0) \leq  \delta^{-d-6} \partial_{11} \psi_{\tau-1}(0),
\]
which completes the proof.
\end{proof}
We note that \eqref{eqnMaxEigUniformness} will hold if, for example, the dataset is quasi-uniform, which means that $h_j/q_j$ is bounded above by a constant.

Now with \eqref{eqnCondNumberMaxMinEig} and Theorems \ref{thmMinEig1} and \ref{thmMaxEig}, we obtain the following theorem where we write $q_j := q_{X_j}$.

\begin{thm}
\label{thmCondNumber}
Suppose the kernel $\mathbf{\Phi}_{\delta}$ is defined as in Theorem \ref{thmMinEig1}. Then the condition number of the multiscale symmetric collocation matrix in Algorithm \ref{AlgSymmStokes} is level-dependent and is bounded by
\[
\kappa_j \leq C \, \left( \frac{\bar{h}_j}{q_j}\right)^{2\tau-d} \bar{h}_j^{-\frac{3}{\tau+1}(2\tau-d) -d },
\]
if $\delta \geq 1$ and by
\[
\kappa_j \leq C \, \left( \frac{\bar{h}_j}{q_j}\right)^{2\tau-d} \bar{h}_j^{-\frac{3}{\tau+1}(2\tau-d-4) -d-4 },
\]
if $\delta < 1$. In the case of quasi-uniform datasets and $h_j \leq 1$, these reduce to
\[
\kappa_j \leq C \, \bar{h}_j^{-2 \tau}.
\]
\end{thm}
\begin{proof}
The first two results follows with $\delta_j = \beta \bar{h}_j^{1-3/(\tau+1)}$ and \eqref{eqnCondNumberMaxMinEig} and Theorems \ref{thmMinEig1} and \ref{thmMaxEig}. If the datasets are quasi-uniform, which means that $h_j/q_j$ is bounded above by a constant, the final result follows by simplifying the first two expressions.
\end{proof}

\section{Numerical experiments} \label{SectionNumExperiments}

In this section, we present the results from applying the multiscale algorithm described in Algorithm \ref{AlgSymmStokes} with $\Omega = [0,1]^2$ and $\nu = 1$ to the Stokes problem  with exact solution given by
\begin{equation*}
\mathbf{u}(x_1,x_2) = \left(
\begin{array}{r}
2\cos(5x_1) \cos(2x_2) \\
5\sin(5x_1) \sin(x_2)
\end{array} \right),
\end{equation*}
\[
\quad p(x_1,x_2) = \sin(3x_1)\sin(3x_2) + C.
\]
This gives
\begin{equation*}
\mathbf{f}(x_1,x_2) = \left(
\begin{array}{r}
58\cos(5x_1)\cos(2x_2) + 3\cos(3x_1)\sin(3x_2) \\
145\sin(5x_1)\sin(2x_2)+3\sin(3x_1)\cos(3x_2)
\end{array} \right)
\end{equation*}
and $\mathbf{g}$ equal to the restriction of $\mathbf{u}(\mathbf{x})$ to $\partial \Omega$ .

We use the $C^8$ Wendland radial basis function given by
$$\psi_{6,4}(\| \mathbf{x} \|) = \left(1-\|\mathbf{x}\|\right)^{10}_+ \, \left(429\|\mathbf{x}\|^4 + 450\|\mathbf{x}\|^3 + 210\|\mathbf{x}\|^2 + 50\|\mathbf{x}\|+5 \right),$$
which is positive definite on $\mathbb{R}^2$ and generates the Sobolev space $H^{5.5}(\mathbb{R}^2)$ \cite{Wen05}. We use the same kernel for both $\psi_{\tau+1}$ and $\psi_{\tau-1}$. Consequently, in this case $\tau = 4.5$.
Since $d=2$, our approximate solution takes the form
\begin{multline*}
\mathbf{S}_{X}\mathbf{v}(\mathbf{x}) = \sum_{j=1}^N \alpha_{1,j} \begin{pmatrix} \nu \partial_{22} \Delta \psi_{\tau+1}(\mathbf{x}-\mathbf{x}_j) \\ -\nu \partial_{12} \Delta \psi_{\tau+1}(\mathbf{x}-\mathbf{x}_j) \\ - \partial_{1} \psi_{\tau-1}(\mathbf{x}-\mathbf{x}_j) \end{pmatrix} + \sum_{j=N+1}^M \alpha_{1,j} \begin{pmatrix} -\partial_{22} \psi_{\tau+1}(\mathbf{x}-\mathbf{x}_j) \\  \partial_{12} \psi_{\tau+1}(\mathbf{x}-\mathbf{x}_j) \\ 0 \end{pmatrix} \\ + \sum_{j=1}^N \alpha_{2,j} \begin{pmatrix} -\nu \partial_{12} \Delta \psi_{\tau+1}(\mathbf{x}-\mathbf{x}_j) \\ \nu \partial_{11} \Delta \psi_{\tau+1}(\mathbf{x}-\mathbf{x}_j) \\ - \partial_{2} \psi_{\tau-1}(\mathbf{x}-\mathbf{x}_j) \end{pmatrix}  + \sum_{j=N+1}^M \alpha_{2,j} \begin{pmatrix} \partial_{12} \psi_{\tau+1}(\mathbf{x}-\mathbf{x}_j) \\  -\partial_{11} \psi_{\tau+1}(\mathbf{x}-\mathbf{x}_j) \\ 0 \end{pmatrix}.
\end{multline*}
We used five levels for the approximation, with $N$ equally spaced points for the interior point sets and $4(\sqrt{N}-1)$ equally spaced boundary centres. The number of interior points, $N_j$, the number of boundary points, $M_j-N_j$, and the maximum mesh norms at each level, $\bar{h}_j$, are given in Table \ref{tblCollMeshNorms}. We note that the (maximum) mesh norms decrease by one half at each level and hence we select $\mu = \frac{1}{2}$.
\begin{table}[!htbp]
\begin{centering}
\small
\begin{tabular}{|c|c|c|c|c|c|}
\hline
Level&1&2&3&4&5\\
\hline
$N$ & 25  & 81 & 289 & 1089 & 4225 \\
$M-N$ & 16 & 32 & 64 & 128 & 256 \\
$\bar{h}$ & 1/4  & 1/8 & 1/16 & 1/32 & 1/64 \\
\hline
\end{tabular} \caption{The number of interior and boundary points used at each level and the maximum mesh norm each level for the numerical experiment} \label{tblCollMeshNorms}
\end{centering}
\end{table}
For the scaling parameters, since $\tau = 4.5$, Algorithm \ref{AlgSymmStokes} specifies that $$\delta_j = \beta \bar{h}_j^{2.5/5.5}$$ with $\beta$ constant. With the given value of $\bar{h}_1$ in Table \ref{tblCollMeshNorms}, we select $\beta$ such that $\delta_1 = 10$. This gives $\beta = 18.779 $ and we use this to generate the other $\delta$ values which are given along with the $L_2$ and $L_{\infty}$ errors in Table \ref{tblStokesExample}. The $L_2$ error was estimated using Gaussian quadrature with a $300 \times 300$ tensor product grid of Gauss-Legendre points and the $L_{\infty}$ error was estimated with the same tensor product grid. We used MATLAB for the calculations and worked with double precision.
\begin{table}[!htbp]
\begin{centering}
\small
\begin{tabular}{|c|c|c|c|c|c|}
\hline
Level&1&2&3&4&5\\
\hline
$\delta_j$ & 10 & 7.29 & 5.33 & 3.89 & 2.84 \\
$\|\mathbf{e}_{\mathbf{u},j}\|_{\mathbf{L}_2(\Omega)}$ & 1.592e{-02} & 6.498e{-04} & 3.274e{-05} & 1.650e{-06} & 1.028e{-07} \\
$\|\mathbf{e}_{\mathbf{u},j}\|_{\mathbf{L}_{\infty}(\Omega)}$ & 2.740e{-02} & 2.233e{-03} & 1.462e{-04} & 8.268e{-06} & 4.579e{-07} \\
$\|\nabla e_{p,j}\|_{L_2(\Omega)}$ & 1.112e{+00} & 1.222e{-01} & 1.235e{-02} & 2.561e{-03} & 5.612e{-04} \\
$\|\nabla e_{p,j}\|_{L_{\infty}(\Omega)}$ & 4.209e{+00} & 3.338e{-01} & 1.048e{-01} & 3.650e{-02} & 1.211e{-02} \\
\hline
\end{tabular} \caption{The scaling factors and approximation errors of the collocation matrices for the multiscale symmetric collocation Stokes problem example} \label{tblStokesExample}
\end{centering}
\end{table}

\bibliographystyle{spmpsci}
\bibliography{E:/Academic/PhD/LaTex/Bibliography/AndrewChernihPhDthesis}{}
\end{document}